\documentclass[11pt]{amsart}

\usepackage{amssymb, amsthm, amsmath, gensymb}
\usepackage{graphicx, comment, accents}
\usepackage[small]{caption}
\usepackage{subcaption}
\usepackage{epsfig}
\usepackage{tikz, pgfplots, float}

\newcommand{\later}[1]{}
\newcommand{\old}[1]{}

\usepackage{amsfonts}

\usepackage[utf8]{inputenc}
\usepackage{fullpage}
\usepackage{framed}
\usepackage{multirow}
\usepackage{enumerate}
\usepackage{url}
\usepackage[breaklinks]{hyperref}
\usepackage{cleveref}
\hypersetup{
	colorlinks = true, 
	urlcolor = cyan, 
	linkcolor = teal, 
	citecolor = cyan 
}

\newcommand{\LA}{\mathop{}\!\mathrm{La^*}}
\newcommand{\LAc}{\mathop{}\!\mathrm{La^*_{con}}}
\newcommand{\LAw}{\mathop{}\!\mathrm{La}}
\newcommand{\LAwc}{\mathop{}\!\mathrm{La_{con}}}

\newcommand{\ar}{\mathop{}\!\mathrm{ar^*}}
\newcommand{\arw}{\mathop{}\!\mathrm{ar}}
\newtheorem{thm}{Theorem}[section]

\newtheorem{definition}[thm]{Definition}
\newtheorem{lem}[thm]{Lemma}

\newtheorem{claim}[thm]{Claim}

\newtheorem{prop}[thm]{Proposition}
\theoremstyle{definition}

\newtheorem{prob}[thm]{Problem}

\newtheorem{obs}[thm]{Observation}

\newcommand{\bC}{{\mathbf C}}
\newcommand{\cC}{{\mathcal C}}
\newcommand{\cQ}{{\mathcal Q}}
\newcommand{\cL}{{\mathcal L}}
\newcommand{\cP}{{\mathcal P}}
\newcommand{\cF}{{\mathcal F}}

\newcommand{\cG}{{\mathcal G}}
\newcommand{\cD}{\boldsymbol{\mathcal D}}

\newcommand{\cS}{{\mathcal S}}

\newcommand{\cU}{{\mathcal U}}

\title{Anti-Ramssey forbidden poset problems}

\author{ Bal\'azs Patk\'os}
\address{HUN-REN Alfr\'ed R\'enyi Institute of Mathematics} 
\email{patkos@renyi.hu}
\date{}

\begin{document}

\maketitle
\begin{abstract}
    A family $\cG$ of sets is a weak copy of a poset $P$ if there is a bijection $f:P\rightarrow \cG$ such that $p\leqslant q$ implies $f(p)\subseteq f(q)$. If $f$ satisfies $p\leqslant q$ if and only if $f(p)\subseteq f(q)$, then $\cG$ is a strong copy of $P$. We study the anti-Ramsey numbers $\arw(n,P), \ar(n,P)$, the maximum number of colors used in a coloring of $2^{[n]}$ that does not admit a rainbow weak or strong copy of $P$, respectively. We establish connections to the well-studied extremal numbers $\LAw(n,P)$ and $\LA(n,P)$ and determine asymptotically $\ar(n,T)$ for all tree posets $T$ and $\ar(n,O_{2k})$ for all crown posets $O_{2k}$.
\end{abstract}

\begin{flushright}
    Dedicated to Gyula O.H. Katona on his 85th birthday.
\end{flushright}

\section{Introduction}

In extremal combinatorics, Tur\'an type problems ask for determining the maximum possible size of a combinatorial object that does not contain a fixed forbidden object. An object is \textit{rainbow} with respect to a coloring if all its elements receive distinct colors. Rainbow Tur\'an problems (introduced for graphs by Keevash, Mubayi, Sudakov, and Verstraete \cite{KMSV}) ask for determining the maximum possible size of an object that does not admit rainbow forbidden objects with respect to any 'proper' coloring. Anti-Ramsey problems (introduced for graphs by Erd\H os, Simonovits, and S\'os \cite{ESS}) constitute another variant of rainbow problems. They ask to maximize the number of colors used when coloring a big object under the condition that no rainbow forbidden small object appears.

We will be interested in the latter type of problems in the area of forbidden subposets. A family $\cG$ of sets is a \textit{weak copy} of a poset $(P,\leqslant)$, if there exists a bijection $f:P\rightarrow \cG$ such that $p\leqslant q$ implies $f(p)\subseteq f(q)$, and $\cG$ is a strong copy of $(P,\leqslant)$, if there exists a bijection $f:P\rightarrow \cG$ such that $p\leqslant q$ if and only if $f(p)\subseteq f(q)$. With some abuse of notation, we will write $P$ instead of $(P,\leqslant)$ and say that a family $\cF$ is \textit{weak (strong) $P$-free} if it does not contain weak (strong) copies of $P$. The \textit{forbidden subposet problem} of determining $\LAw(n,P)$ $(\LA(n,P))$, the maximum size of a weak (strong) $P$-free family $\cF\subseteq 2^{[n]}:=\{S:S\subseteq \{1,2,\dots,n\}\}$ was introduced by Katona and Tarj\'an \cite{KT} in the early 80s. For results and open problems, see the survey \cite{GLsurv} and Chapter 7 of \cite{GP}.

Rainbow Tur\'an forbidden problems were introduced and studied in \cite{Prain}, and another rainbow variant, Gallai-Ramsey numbers of poset pairs, was studied in \cite{Ketal}. In this paper, we start investigating the anti-Ramsey version of the forbidden subposet problem. Let $\arw(n,P)$ ($\ar(n,P)$) denote the maximum number of colors used in a coloring of $2^{[n]}$ that does not admit a rainbow weak (strong) copy of $P$. By definition, we immediately obtain the inequalities \begin{equation}\label{triv}
    \arw(n,P)\le \LAw(n,P) \hskip 1truecm \text{and} \hskip 1truecm \ar(n,P)\le \LA(n,P),
\end{equation} as if a coloring of $2^{[n]}$ uses more than $\LAw(n,P)$ ($\LA(n,P)$) colors, then the family we get by picking an arbitrary set from each color class contains a weak (strong) copy of $P$ which is rainbow. The first inequality of (\ref{triv}) and $\LAw(n,C_|P|)\le (|P|-1)\binom{n}{\lfloor n/2\rfloor}$, an old result of Erd\H os \cite{E} imply $\arw(n,P)=O(\binom{n}{\lfloor n/2\rfloor})$ and the second inequality of (\ref{triv}) together with a result of Methuku and P\'alv\"olgyi \cite{MP} imply $\ar(n,P)=O(\binom{n}{\lfloor n/2\rfloor})$ 

\bigskip

Let us make further quick observations that improve the inequalities of (\ref{triv}) and that give lower bounds on $\arw(n,P)$ and $\ar(n,P)$ and still connect the anti-Ramsey number to the original forbidden subposet problem. We need two concepts.
A family $\cF$ is said to be \textit{convex} if $F\subseteq G\subseteq F'$ and $F,F'\in \cF$ imply $G\in \cF$. Let $\LAwc(n,P)$  ($\LAc(n,P)$) denote the maximum size of a weak (strong) $P$-free convex family $\cF\subseteq 2^{[n]}$.

For any poset $P$, we write $\cP^-(P)$ to denote the set $\{P\setminus \{m\}: m ~\text{is maximal or minimal in}\ P\}$. Also, for a set $\cP$ of posets, we say that the family $\cF$ is weak (strong) $\cP$-free if it is weak (strong) $P$-free for all $P\in \cP$. $\LAw(n,\cP)$ (respectively $\LA(n,\cP))$ denotes the maximum size of a weak (strong) $\cP$-free family $\cF\subseteq 2^{[n]}$.

\begin{prop}\label{lowertriv}
    For any poset $P$, we have
    \begin{enumerate}
        \item 
        $1+\LAwc(n,\cP^-(P))\le \arw(n,P)$, 
        \item 
        $1+\LAc(n,\cP^-(P))\le \ar(n,P)$.
    \end{enumerate}
\end{prop}

\begin{proof}
    Let $\cF\subseteq 2^{[n]}$ be a convex weak (strong) $\cP^-(P)$-free family of size $\LAwc(n,\cP^-(P))$ (respecitvely $\LAc(n,\cP^-(P))$). Color each $F\in \cF$ with distinct colors and use the same new color for all sets in $2^{[n]}\setminus \cF$. By the convexity of $\cF$, sets of the new color can only be maximal or minimal elements in a weak (strong) rainbow copy of $P$, so the remaining sets of the copy should form a weak (strong) copy of a poset in $\cP^-(P)$ within $\cF$. 
\end{proof}

As a corollary, we immediately obtain the following. Let $C_k$ denote the chain (totally ordered set) on $k$ elements. Then if $P$ contains a weak copy of $2C_2$ (two unrelated copies of $C_2$) or a copy of $C_3$, then $\ar(n,P),\arw(n,P)=\Theta(\binom{n}{\lfloor n/2\rfloor})$. Indeed, then any $P'\in \cP^-(P)$ contains a pair of elements in relation, so both $\LAc(n,\cP^-(P))$ and $\LAwc(n,\cP^-(P))$ are at least $\binom{n}{\lfloor n/2\rfloor}$ as shown by the convex family of the middle layer of $2^{[n]}$.

\begin{prop}\label{uppertriv} \
    \begin{enumerate}
        \item 
        For any poset $P$, we have $\arw(n,P)\le 2+\LAw(n,\cP^-(P))$.
        \item 
        If a poset $P$ contains a smallest or largest element $m$, then $\ar(n,P)\le 1+\LA(n,P\setminus \{m\})$.
    \end{enumerate}
\end{prop}

\begin{proof}
    For any coloring $c$ of $2^{[n]}$ without a rainbow weak copy of $P$, pick one set from every color class of $c$ except for the color classes of $\emptyset$ and $[n]$. The family obtained must be $\cP^-(P)$-free, as such a copy is rainbow by definition, and can be extended by $\emptyset$ or by $[n]$ to obtain a rainbow weak copy of $P$. This proves (1).

    If $m$ is the largest element of $P$, then for any coloring $c$ of $2^{[n]}$ without a rainbow strong copy of $P$, pick one set from every color class of $c$ except for the color class of $[n]$. The family obtained must be $(P\setminus \{m\})$-free, as such a copy can be extended by $[n]$ to obtain a rainbow strong copy of $P$. This proves (2).
\end{proof}

As a corollary, we obtain that if $P$ does not contain a weak copy of $2C_2$ nor a copy of $C_3$, then $\arw(n,P)=\min\{|P|-1\}$. Indeed, then $\cP^-(P)$ contains the antichain $A_k$ on $k=|P|-1$ elements, and any $k$ sets form a \textit{weak} copy of $A_k$, so $\LAw(n,A_k)=\min\{k-1,n\}$.

All weak $\{2C_2,C_3\}$-free posets $P$ are of the form $P'+A_k$, where $+$ denotes the union of an incomparable pair of copies of $P'$ and $A_k$ with $P'$ being either a fork $\vee_s$ ($a<b_1,b_2,\dots,b_s$) or a broom $\wedge_s$ ($c_1,\dots,c_s<d$). If $s=0,k=1$, then $\vee_s+A_k=A_1$ is the one element poset, and no coloring can avoid a rainbow strong copy of $A_1$. If $s=1,k=0$, then $\vee_s+A_k=C_2$ and to avoid a rainbow strong copy of $C_2$, all sets should receive the color of the empty set, so the coloring can use only 1 color. For all other values of $s$ and $k$, we have the following  proposition that is proved at the end of the paper.

\begin{prop}\label{noC2}\
    \begin{enumerate}
        \item 
       $\ar(n,\vee_s)=\ar(n,\wedge_s)=(s-1)(n-1)+2$ for any $s\ge 2$.
        \item 
        $\ar(n,A_k)=3+(k-2)(n-1)$  if $n\ge 2k$  for any $k\ge 2$.
        \item $\ar(n,\wedge_s+A_k)=\ar(n,\vee_s+A_k)=O_{s,k}(n^2)$ for any $s,k\ge 1$.
     \end{enumerate}
\end{prop}

\medskip

For weak copies of $P$, the lower bound of Proposition \ref{lowertriv} and the upper bound of Proposition \ref{uppertriv} differ only by 1 and the possible difference of $\LAwc(n,P)$ and $\LAw(n,P)$. In most cases, these are negligible. For example, if $P=\Diamond$ is the diamond poset on four elements $a,b,c,d$ with $a<b,c<d$ being all relations, then $\cP^-(\Diamond)=\{\vee, \wedge\}$. A result of Katona and Tarján \cite{KT} states that $\LAw(n,\{\vee,\wedge\})=2\binom{n-1}{\lfloor \frac{n-1}{2}\rfloor}$ and an extremal family is $\binom{[n-1]}{\lfloor \frac{n-1}{2}\rfloor}\cup \{F\cup \{n\}:F\in \binom{[n-1]}{\lfloor \frac{n-1}{2}\rfloor}\}$, which is convex. (Here and throughout the paper, $\binom{S}{k}$ denotes the family of all $k$-subsets of $S$, and $\binom{[n]}{k}$ is often referred to as the $k$th layer of $2^{[n]}$.) So, $\LAw(n,\cP^-(\Diamond))=\LAwc(n,\cP^-(\Diamond))$. Finally, if in a coloring $c$ of $2^{[n]}$ the empty set and $[n]$ receive distinct colors, then to avoid a rainbow weak copy of $\Diamond$, one can only use at most one further color. Therefore, $\arw(n,\Diamond)=1+\LAw(n,\cP^{-}(\Diamond))=1+\LAwc(n,\cP^{-}(\Diamond))=1+2\binom{n-1}{\lfloor \frac{n-1}{2}\rfloor}$.

How far apart can $\LAw(n,\cP^{-}(P))$ and $\LAwc(n,\cP^{-}(P))$ be? It was conjectured \cite{B,GLu} that for any set $\cP$ of posets, $\LAw(n,\cP)$ ($\LA(n,\cP)$) is asymptotically the size of most number of middle layers of $2^{[n]}$ that one can have without containing of weak (strong) copy of a member of $\cP$. Note that the union of consecutive layers is always a convex family. The conjecture was disproved by Ellis, Ivan, and Leader \cite{EIL}, who showed that $\LA(n,B_d),\LAw(n,B_d)\ge (d+\varepsilon-o(1))\binom{n}{\lfloor \frac{n}{2}\rfloor}$ for $d\ge 4$, where $B_d$ is the Boolean poset on $2^{[d]}$ ordered by inclusion. Although the construction of \cite{EIL} for $B_d$ can be easily turned into a construction for $B_d\setminus \{\emptyset\}$ or for $B_d\setminus \{[d]\}$, we do not see how it could be turned into a construction that avoids both of them, i.e. a construction that is $\cP^-(B_d)$-free. 

\begin{prob}
    Is there a poset $P$ and a positive real $\varepsilon$ such that $\LAw(n,\cP^-(P))\ge (1+\varepsilon)\LAwc(n,\cP^-(P))$?
\end{prob}

\subsection{Main results}
We will focus on forbidding strong copies of $P$ and on cases not covered by Proposition \ref{uppertriv} (2), i.e. on posets having neither largest nor smallest elements.

The oriented Hasse diagram $\overrightarrow{H}(P)$ of a poset $P$ is an oriented graph with vertex set $P$ where $\overrightarrow{pq}$ is an arc if $q$ covers $p$, i.e. $p\leqslant_P q$ and there is no $z\neq p,q$ with $p\leqslant_P z\leqslant_P q$. We obtain the (unoriented) Hasse diagram $H(P)$ of $P$ from $\overrightarrow{H}(P)$ by dropping the orientation of all arcs. A poset is a \textit{tree poset} if $H(P)$ is a tree.

Bukh proved \cite{B} $\LAw(n,T)=(h(T)-1+o(1))\binom{n}{\lfloor n/2\rfloor}$ for any tree poset $T$, where $h(T)$ denotes the \textit{height of $T$}, the length of the longest chain in $T$ or equivalently the most number of vertices in a directed path of the directed Hasse diagram. This was strengthened by Boehnlein and Jiang \cite{BJ} to $\LA(n,T)=(h(T)-1+o(1))\binom{n}{\lfloor n/2\rfloor}$. As the construction showing the lower bound in these results, consists of the middle $h(T)-1$ layers, which is a convex family, we obtain $\LAwc(n,T)=(1+o(1))\LAw(n,T)=(1+o(1))\LAc(n,T)=(1+o(1))\LA(n,T)$ for all tree posets. Our first theorem determines $\ar(n,T)$ asymptotically for all tree posets. 

Note that if for all maximal or minimal elements of $T$ we have $h(T\setminus \{t\})=h(T)$, then the result follows from $\ar(n,T)\le \LA(n,T)$ of (\ref{triv}) and $1+\LA(n,\cP^-(T))\le \ar(n,T)$ of Proposition \ref{lowertriv} (2). Also, if $T$ has a largest or a smallest element, then the theorem follows from Proposition \ref{lowertriv} (2) and Proposition \ref{uppertriv} (2). So the hard part of the result is about tree posets $T$ containing a minimal or maximal element $m$ that is contained in all longest chains of $T$ and yet $m$ is neither smallest nor largest element of $T$. For example, in the poset on elements $a,b,c,d,e$ with $a\prec c, a\prec d, b\prec d, d\prec e$ being all cover relations, $e$ is contained in both chains of legth 3, namely $a\prec d\prec e, b\prec d\prec e$, but $e$ is not a largest element as $c\not\prec e$. Note that Proposition ~\ref{noC2} covers all tree posets for which $\ar(n,T)=o(\binom{n}{\lfloor n/2\rfloor})$, so the following theorem holds only for all other tree posets.

\begin{thm}\label{tree}
    For any tree poset $T$ with $T\neq \vee_s,\wedge_s$ for any $s$, we have $\ar(n,T)=(1+o(1))\LA(n,\cP^-(T))$.
\end{thm}

The \textit{crown poset} $O_{2k}$ is a poset with oriented Hasse diagram an antidirected cycle, i.e. $O_{2k}$ has elements $a_1,a_2,\dots,a_k,b_1,b_2,\dots,b_k$ with $a_i,a_{i+1}<b_i$ for all $i=1,2,\dots,k-1$ and $a_k,a_1<b_k$. $O_4$ is mostly called the butterfly poset and is denoted by $\bowtie$ based on the drawing of its Hasse diagram. $\LAw(n,\bowtie)=\binom{n}{\lfloor \frac{n}{2}\rfloor}+\binom{n}{\lfloor \frac{n}{2}\rfloor+1}$ was proved first by DeBonis, Katona, and Swanepoel \cite{DKS}. As $\bowtie$ is a  strong subposet of the tree poset $X$ on elements $a_1,a_2,b_1,b_2,c$ with $a_1,a_2<c<b_1,b_2$, Boehnlein and Jiang's result  implies $\LA(n,\bowtie)=(2+o(1))\binom{n}{\lfloor \frac{n}{2}\rfloor}$. For even $k\ge 4$, $\LAw(n,O_{2k})=(1+o(1))\binom{n}{\lfloor \frac{n}{2}\rfloor}$ was proved by Griggs and Lu \cite{GLu}, and the same result for odd $k\geq 7$ was obtained by Lu \cite{Lu}. The value of $\LA(n,O_{2k})$ for $k\ge 3$ has not yet been addressed.

As $\cP^-(O_{2k})$ consists of two tree posets both of height 2, Proposition \ref{lowertriv} (1) and Proposition \ref{uppertriv} (1) imply $\arw(n,O_{2k})=(1+o(1))\binom{n}{\lfloor n/2\rfloor}$ for all $k\ge 2$. By (\ref{triv}), we have $\ar(n,\bowtie)\le \LA(n,\bowtie)=(2+o(1))\binom{n}{\lfloor \frac{n}{2}\rfloor}$. It is easy to see that $\ar(n,\bowtie)\ge \binom{n}{\lfloor \frac{n}{2}\rfloor}+\binom{n}{\lfloor \frac{n}{2}\rfloor+1}+1$. Indeed, color sets of size $\lfloor \frac{n}{2}\rfloor$ and $\lfloor \frac{n}{2}\rfloor +1$ by their own color, and color all other sets white. As the intersection of two unrelated sets of different colors has size at most $\lfloor \frac{n}{2}\rfloor$ and the union of two unrelated sets of different colors has size at least $\lfloor \frac{n}{2}\rfloor+1$, this coloring does not admit a strong rainbow copy of $\bowtie$. This shows $\ar(n,\bowtie)=(2+o(1))\binom{n}{\lfloor \frac{n}{2}\rfloor}$. As $\cP^-(\bowtie)=\{\vee,\wedge\}$, this also shows an example when the lower bound of Proposition \ref{lowertriv} (2) and the trivial upper bound (\ref{triv}) differ asymptotically and the true asymptotics of $\ar(n,P)$ is given by (\ref{triv}). For all values $k\ge 3$ we show the following result.

\begin{thm}\label{crown}
    For any $k\ge 3$, we have $\ar(n,O_{2k})=(1+o(1))\binom{n}{\lfloor \frac{n}{2}\rfloor}$.
\end{thm}

After submitting our manuscript, we learned from Maria Axenovich that Jacob Manske in his PhD thesis had already
introduced anti-Ramsey problems for posets and obtained some preliminary results. The only overlap between his and
our results is Proposition \ref{noC2} (2).

\section{Proofs}

In the next subsection, we state some embedding results for tree posets and how they imply Theorem \ref{tree} and Theorem \ref{crown}. Then in Section 2.2, we prove the embedding results. Finally, Section 2.3 contains the proof of Proposition \ref{noC2}.

\subsection{Statements of embedding results and their implications}

To be able to state our results, we need to introduce some notation. 

Let $\tilde{B}_n=\{F\in 2^{[n]}:||F|-n/2|\le 2\sqrt{n\ln n}\}$ and $\tilde{B}_n^T=\{F\in 2^{[n]}:||F|-n/2|\le 4|T|\sqrt{n\ln n}\}$. Note that by Chernoff's inequality, $|2^{[n]}\setminus \tilde{B}_n^T|\le|2^{[n]}\setminus \tilde{B}_n|\le \frac{1}{n^2}\binom{n}{\lfloor \frac{n}{2}\rfloor}$ holds for large enough $n$.

For any poset $P$ and $p\in P$, we set $\cD_P(p)=\{q\in P:q\leqslant_P p\}$. Similarly, for any family $\cF\subseteq 2^{[n]}$ of sets and $G\in 2^{[n]}$, we write $\cD_\cF(G)=\{F\in \cF: F\subseteq G\}$ and $\cD(G)=\{F:F\subseteq G\}$, $\cU_n(G)=\{F\in 2^{[n]}:F\supseteq G\}$ and we omit $n$ from the subscript when it is clear from context.

Finally, if $\cG$ is a strong copy of a poset $P$ shown by the bijection $f:P\rightarrow \cG$, then the set $f(p)\in \cG$ will be denoted by $G_p$. 

\begin{thm}\label{specialtree}
    Let $T$ be a tree poset of height $k+1$ that contains a maximal element $m$ such that all chains of length $k+1$ in $T$ contain $m$. Then for any positive real $\varepsilon$, there exists $n_0=n_0(T,\varepsilon)$ such that if $n\ge n_0$ and $\cF\subseteq \tilde{B}_n$ with $|\cF|= (k-1+\varepsilon)\binom{n}{\lfloor n/2\rfloor}$, then $\cF$ contains a strong copy $\cG$ of $T\setminus \{m\}$ such that for any $p\in T\setminus \cD_T(m)$ the image $G_p$ is not a subset of $\cup_{d\in \cD_T(m)\setminus \{m\}}G_d$.
\end{thm}

Let $P_{2k-1}$ denote the poset that we obtain from $O_{2k}$ by removing a maximal element, i.e. $P_{2k-1}$ has $2k-1$ elements $a_1,a_2,\dots,a_k,b_1,b_2,\dots,b_{k-1}$ with $a_i,a_{i+1}<b_i$.

\begin{thm}\label{crownemb}
For any $k\ge 3$ and positive real $\varepsilon$, there exists $n_0=n_0(\varepsilon)$ such that if $n\ge n_0$ and $\cF\subseteq \tilde{B}_n$ with $|\cF|= (1+\varepsilon)\binom{n}{\lfloor n/2\rfloor}$, then $\cF$ contains a strong copy $\cG$ of $P_{2k-1}$ such that $G_{a_i}\not\subseteq G_{a_1}\cup G_{a_k}$ for all $i=2,\dots,k-1$.
\end{thm}



\begin{proof}[Proof of Theorem \ref{tree} and Theorem \ref{crown}]
    The lower bounds of both theorems follow from Proposition \ref{lowertriv} (2).

    For the upper bounds let $c$ be any coloring of $2^{[n]}$. Let $\cF_c$ be a family that we obtain by picking one set from each color class that does not contain any set from $2^{[n]}\setminus \tilde{B}_n$. Clearly, the number of colors used by $c$ is at most $|2^{[n]}\setminus \tilde{B}_n|+|\cF_c|\le \frac{1}{n^2}\binom{n}{\lfloor \frac{n}{2}\rfloor}+|\cF_c|$. 
    
    For Theorem \ref{crown}: if $|\cF_c|\le (1+o(1))\binom{n}{\lfloor \frac{n}{2}\rfloor}$, then we are done. Otherwise, by Theorem \ref{crownemb}, we obtain a strong copy $A_1,A_2,\dots, A_k,B_1,\dots,B_{k-1}$ of $P_{2k-1}$ such that $A_1\cup A_k$ does not contain any of the $A_i$s ($i=2,3,\dots,k-1$).
    Let $B_k$ be any superset of $A_1\cup A_k$ of size at least $\frac{n}{2}+2\sqrt{n\ln n}$ with $A_i\not\subseteq B_k$ for all $i=2,3,\dots,k-1$. Such a set exists as there is $a_i\in A_i\setminus (A_1\cup A_k)$ for all $i$, so $[n]\setminus \{a_2,\dots,a_{k-1}\}$ can play the role of $B_k$.
    Then $A_1,A_2,\dots, A_k,B_1,\dots,B_{k-1},B_k$ form a strong copy of $O_{2k}$. Also, it is rainbow as the $A_i$s, $B_i$s represent different color classes of $c$, and these color classes do not contain sets from $2^{[n]}\setminus \tilde{B}_n$, so their colors are distinct from that of $B_k$.

    For  Theorem \ref{tree}: as observed in the introduction, it is enough to consider tree posets $T$ of height $k+1$ with a maximal or minimal element $m$ that is contained in all chains of length $k+1$ in $T$. Without loss of generality, we may assume that $m$ is maximal. If $|\cF_c|\le (k-1+o(1))\binom{n}{\lfloor \frac{n}{2}\rfloor}$, then we are done. Otherwise, by Theorem \ref{specialtree}, $\cF$ contains a strong copy $\cG$ of $T\setminus \{m\}$ such that $G:=\cup_{d\in \cD_T(m)\setminus \{m\}}G_d$ does not contain any $G_p$ with $p\notin \cD_T(m)$. Then a superset $G'$ of $G$ of size at least $\frac{n}{2}+2\sqrt{n\ln n}$ still not containing $G_p$ for any $p\notin \cD_T(m)$ together with $\cG$ form a strong copy of $T$. The existence of $G'$ and the rainbow property of $\cG\cup\{G'\}$ follows as in the previous paragraph.
\end{proof}

\subsection{Proofs of embedding results}
There have been several papers on embedding tree posets into set families starting with the seminal paper of Bukh \cite{B} determining $\LAw(n,T)$ for any tree poset $T$. This was extended by Boehnlein and Jiang \cite{BJ} to strong copies of $T$. The proof of \ref{specialtree} will closely follow the reasoning of \cite{BJ}. Supersaturation and counting results on weak and strong copies of tree posets (for all or for some special classes) were obtained in \cite{PT} and then in \cite{Betal} and \cite{Jetal}. The proof of \ref{crownemb} uses some ideas from \cite{PT}.

\bigskip

We will first prove Theorem \ref{crownemb}. We start by introducing some notation and preliminaries.

The \textit{Lubell-mass} of a family $\cF\subseteq 2^{[n]}$ is defined as
\[
\lambda_n(\cF)=\sum_{F\in \cF}\frac{1}{\binom{n}{|F|}}=\frac{1}{n!}\sum_{\cC\in \mathbf{C}_n}|\cC\cap \cF|,
\]
the average number of sets in $\cF$ that a maximal chain contains chosen uniformly at random from $\mathbf{C}_n$, the set of all maximal chains in $2^{[n]}$. We will consider the max-partition of $\bC_n$ according to $\cF$, that is we write $\bC_F=\{\cC\in \bC_n: F ~\text{ is largest in}\ \cF\cap \cC\}$. Then $\frac{1}{|\bC_F|}\sum_{\cC\in \bC_{F}}|\cC\cap \cF|=\lambda_{|F|}(\cD_{\cF}(F))$ and $\lambda_n(\cF)$ is a weighted average of $\lambda_{|F|}(\cD_{\cF}(F))$ and so we have the following.

\begin{obs}\label{lubaver}
    If $\lambda_{|F|}(\cD_{\cF}(F))\le B$ for all $F$, then $\lambda_n(\cF)\le B$ holds.
\end{obs}

\begin{lem}[Griggs, Li, Lu, Lemma 3.2 in \cite{GLL}]\label{lubm}
    For any family $\cF\subseteq 2^{[n]}$, we have $|\cF|\le \lambda_n(\cF)\binom{n}{\lfloor n/2\rfloor}$.
\end{lem}

We will need the following folklore observation on minimum and average degree of graphs.

\begin{obs}\label{core}
    If $G$ is a graph with average degree $d$, then $G$ contains a subgraph $G'$ with minimum degree at least $d/2$.
\end{obs}

In the proof of Theorem \ref{crownemb}, we will be looking for a copy of a special subposet with the help of which we will find the copy of $P_{2k-1}$ with the extra properties.

\begin{definition}
    Let $S^{k,\ell}$ denote the height 2 poset on $\ell\cdot k+1$ elements such that $H(S^{k,\ell})$ is a spider with $\ell$ legs each of length $k$. As the height is 2, there exist two such posets depending on whether the leaves are maximal or minimal elements. We define $S^{k,\ell}$ to be the one where the leaves are maximal elements.
\end{definition}

We move some of the calculations of the proof of Theorem \ref{crownemb} to here.

  \begin{lem}\label{lemma}
       For $k$ fixed, $1000k\le j\le 4\sqrt{n\ln n}$ and $n$ large enough, we have 
       \begin{enumerate}
           \item 
           \[\frac{\binom{n+2\sqrt{n\ln n}}{j}}{\binom{n-2\sqrt{n\ln n}}{j}}\le n^{20} \]
           \item 
           \[\sum_{i=j/k}^{j}\binom{kn^{1/3}j}{i}\binom{n/2+2\sqrt{n\ln n}}{j-i}=o\left(\frac{1}{\sqrt{n\ln n}}\binom{n/2-2\sqrt{n\ln n}}{j-22}\right)\]
           \item 
           \[
           \binom{n^{2/3}+2\sqrt{n\ln n}+kn^{1/3}{j}}{j}=o\left(\frac{1}{\sqrt{n\ln n}}\binom{n/2-2\sqrt{n\ln n}}{j-22}\right)
           \]
       \end{enumerate}
  \end{lem} 

  \begin{proof}Using $1+x\le e^x$, we obtain
    \begin{equation*}
\frac{\binom{n+2\sqrt{n\ln n}}{j}}{\binom{n-2\sqrt{n\ln n}}{j}}\le\left(\frac{n-2\sqrt{n\ln n}}{n-6\sqrt{n \ln n}}\right)^{4\sqrt{n\ln n}}\le e^{\frac{5\sqrt{n\ln n}}{n}\cdot 4\sqrt{n\ln n}}\le n^{20}       
    \end{equation*}
      which proves (1).

      Using (1), $\binom{n}{\ell}\le (\frac{ne}{\ell})^\ell$,  $\frac{\binom{n}{\ell}}{\binom{n}{\ell'}}\le (\frac{\ell'}{n})^{\ell'-\ell}$ for $\ell<\ell'$ and $i\ge j/k\ge 1000$, $j\le \sqrt{n\ln n}$, we obtain
      \[
      \frac{\binom{kn^{1/3}j}{i}\binom{n/2+2\sqrt{n\ln n}}{j-i}}{\binom{n/2-2\sqrt{n\ln n}}{j-22}}\le \frac{n^{20}\binom{kn^{1/3}j}{i}\binom{n/2+2\sqrt{n\ln n}}{j-i}}{\binom{n/2+2\sqrt{n\ln n}}{j-22}}\le \frac{n^{20}(ek^2n^{1/3})^i(2j)^{i-22}}{n^{i-22}}\le n^{42}\left(\frac{2ek^2j}{n^{2/3}}\right)^i=o\left(\frac{1}{n^2}\right).
      \]
    As there are at most $4\sqrt{n\ln n}$ terms in its left hand side, we obtain (2).

    Using $(\frac{n}{\ell})^\ell\le \binom{n}{\ell}\le (\frac{en}{\ell})^\ell$, we have $\binom{n^{2/3}+2\sqrt{n\ln n}+kn^{1/3}{j}}{j}\le (n^{2/3})^j$ and $\binom{n/2-2n\sqrt{n\ln n}}{j-22}\ge (\frac{n}{10\sqrt{n \ln n}})^{j-22}$. As $j\ge 1000k$, we have $(n^{2/3})^j<n^{0.99j-22}$ and (3) follows.
  \end{proof}

    After all these preparations, we are ready to prove Theorem \ref{crownemb}.

\begin{proof}[Proof of Theorem \ref{crownemb}]
    Let $\cF\subseteq \tilde{B}_n$ be a set family with $|\cF|= (1+\varepsilon)\binom{n}{\lfloor n/2\rfloor}$. For any $F\in \cF$ and $j$ positive integer, let us define \[S_j(F)=\{G\in \cF:  G\subseteq F, |G|=|F|-j\}.\] 
    Then we partition $\cF$ into $\cF_1\cup \cF_2\cup \cF_3$ with \[\cF_1=\left\{F\in \cF:\exists j: 1\le j \le 1000k, |S_j(F)|\ge \frac{\varepsilon}{10000k} \binom{|F|}{j}\right\},\] \[\cF_2=\left\{F\in \cF\setminus \cF_1: \exists j \ge  1000k+1,  |S_j(F)|\ge \binom{|F|}{j-22}\right\}, \hskip 0.7truecm \text{and}\ \hskip 0.7truecm\cF_3=\cF\setminus (\cF_1\cup\cF_2).\]

    \medskip

    We obtain the following bound on $\lambda(\cF_3)$ and thus, by Lemma \ref{lubm}, on $|\cF_3|$.  

    \begin{claim}\label{Mclaim}
        If $n$ is large enough, then $\lambda_n(\cF_3)< 1+\varepsilon/2$.
    \end{claim}

    \begin{proof}[Proof of Claim]
        By definition, for any $F\in \cF_3$ 
        \[
        \lambda_{|F|}(\cD_{\cF_3}(F))\le \sum_{j=0}^{4\ln n\sqrt{n}}\frac{|S_j(F)|}{\binom{|F|}{j}}\le 1+1000k\cdot \frac{\varepsilon}{10000k}+4\ln n\sqrt{n}\cdot \left(\frac{4\ln n\sqrt{n}}{n}\right)^{22}\le 1+\varepsilon/2,
        \]
    if $n$ is large enough. The claim follows from Observation \ref{lubaver}. 
    \end{proof}

    By Lemma \ref{lubm}, $|\cF_3|\le (1+\varepsilon/2)\binom{n}{\lfloor n/2\rfloor}$, so either $\cF_1$ or $\cF_2$ has size at least $\frac{\varepsilon}{4}\binom{n}{\lfloor n/2\rfloor}$. On many occasions, we will see that a subfamily of $\cF_1$ or $\cF_2$ will have size $f(\varepsilon)\binom{n}{\lfloor n/2\rfloor}$ or $g(\varepsilon)\binom{n}{j}$. If it is not important, then we will not state what the precise function $f$ is, just write $\varepsilon',\varepsilon''$, etc.

    \medskip

    \textsc{Case I} $|\cF_1|\ge \frac{\varepsilon}{4}\binom{n}{\lfloor n/2\rfloor}$
    
    By definition of $\cF_1$, for any $F\in \cF_1$, there exists $1\le j(F)\le 1000k$ such that $|S_{j(F)}(F)|\ge \frac{\varepsilon}{10000k}\binom{|F|}{j}$. So for some $1\le j\le 1000k$, the family $\cF^j_1=\{F\in \cF_1:j(F)=j\}$ has size at least $\frac{\varepsilon}{4000k}\binom{n}{\lfloor n/2\rfloor}$. Consider the bipartite graph $B$ with parts $\cF$ and $\cF^j_1$ and $F\in \cF,F_1\in \cF^j_1$ joined by an edge if and only if $F\subseteq F_1$, $|F|=|F_1|-j$, and we say that this edge has color set $F_1\setminus F$. By definition of $\cF_1$, the number of edges in $B$ is at least $\varepsilon'\binom{n/2-2\ln n\sqrt{n}}{j}\binom{n}{\lfloor n/2\rfloor}$ and so the average degree of $B$ is at least $\varepsilon''\binom{n/2-2\ln n\sqrt{n}}{j}$. By Observation \ref{core}, $B$ contains a subgraph $B'$ with minimum degree at least $\varepsilon'''\binom{n/2-2\ln n\sqrt{n}}{j}$. 
    
    We plan to obtain a copy of $S^{k-2,\ell}$ in $\cF$ where $\ell=\ell(k,\varepsilon)$ is a large constant independent of $n$. If the element of $S^{k-2,\ell}$ corresponding to the center of the spider is a minimal element, then we start with an arbitrary set in the $\cF$ part of $B'$, if it is a maximal element, then we start with an arbitrary element of the $\cF_1$ part of $B'$. Then we proceed as follows: if some connected part of $S^{k-2,\ell}$ is already defined and we want to add a neighbor of $F$ or $F_1$, then we pick a neighbor such that the color set of the connecting edge should be disjoint with the color set of every previously defined edge of the spider. This is possible, as the union of all previous color sets has size at most $j k \ell$, and so the number of neighbors that do not satisfy this requirement is at most $jk\ell\binom{n/2+2\sqrt{n\ln n}}{j-1}$ which is smaller than the minimum degree $\varepsilon'''\binom{n/2-2\ln n\sqrt{n}}{j}$ if $n$ is large enough (here we use that in Case I, $j$ is a fixed constant at most $1000k$). By definition, this family of sets forms a weak copy of $S^{k-2,\ell}$. We claim that due to the disjointness of the color sets of the edges, they form a strong copy of $S^{k-2,\ell}$. Suppose first that $F\in \cF,F_1\in\cF_1$ are on the same leg of the spider, but are not consecutive elements. If $F$ is closer to the center of the spider, then the leg continues as $FF_1'F'$ and the color set of $F_1'F'$ is removed and never put back, so $F_1$ cannot contain $F$. Similarly, if $F_1$ is closer to the center and the leg continues as $F_1F'F_1'$, then the color set of $F'F_1'$ is added and never later removed, so $F_1$ cannot contain $F$. If $F,F_1$ are on different legs of the spider, then if the center $F_0$ is in the $\cF$ part, then the color set of $F_0F_1'$ is contained in $F$ where $F_1'$ is the first element of the spider after the center, and so $F_1$ cannot contain $F$. Finally, if $F_0$ is in the $\cF_1$ part, then the color set of $F_0F'$ is removed and is never put back, and so is not contained in $F_1$ which therefore cannot contain $F$. Here $F'$ is the first element of the leg of $F_1$. 

    Let $G$ denote the intersection of all sets of the spider defined above. Observe that $|G|\ge n/2-3\sqrt{n\ln n}$ if $n$ is large enough as every set in the spider can remove at most $1000k$ elements from the intersection. Let $F^1,F^2,\dots,F^\ell$ be the sets corresponding to the leaves of the spider and let $x_i$ be an element of $F^i$ from the color set of the edge incident to $F^i$. If we find $1\le a<b\le \ell$ and $y\in G$ such that $F^a$ and $F^b$ have neighbors $G^a\subseteq F^a\setminus \{y\}$ and $G^b\subseteq F^b\setminus \{y\}$ with $x_a\in G^a,x_b\in G^b$, then these two legs of the spider together with $G^a$ and $G^b$ form the desired copy of $P_{2k-1}$.
    
     The number of subsets of $F^a$ of size $|F^a|-j$ not containing $x_a$ is $\binom{|F^a|-1|}{j-1}$, so if $n$ is large enough, then the number of subsets of $F^a$ of size $|F^a|-j$  containing $x_a$ is at least $\varepsilon''''\binom{|F^a|}{j}$. Therefore, for any $F^a$ and $x_a$, there exist $\alpha(\varepsilon)n$ elements $y\in G$ such that there is an $(|F^a|-j)$-subset $G^i_y\in \cF$ with $x_i\in G^i_y,y\notin G^i_y$. Taking $\ell=\ell(\varepsilon)$ to be larger than $\lceil \alpha(\varepsilon)^{-1}\rceil$, we obtain a pair $x_a,x_b$ with a common $y$. This finishes the proof of this case.

    \medskip

    \textsc{Case II} $|\cF_2|\ge \frac{\varepsilon}{4}\binom{n}{\lfloor n/2\rfloor}$

    By definition of $\cF_2$, for any $F\in \cF_2$, there exists $1000k< j(F)\le 4\ln n\sqrt{n}$ such that $|S_{j(F)}(F)|\ge \binom{|F|}{j-22}$. So for some $1000k< j\le  4\sqrt{n\ln n}$, the family $\cF^j_2=\{F\in \cF_2:j(F)=j\}$ has size at least $\frac{\varepsilon}{ 4\sqrt{n\ln n}}\binom{n}{\lfloor n/2\rfloor}$. Consider the bipartite graph $B$ with parts $\cF$ and $\cF^j_2$ and $F\in \cF,F_2\in \cF^j_2$ joined by an edge if and only if $F\subseteq F_2,|F|=|F_2|-j$ and let $F_2\setminus F$ be the color set of this edge. By definition of $\cF_2$, the number of edges in $B$ is at least $\frac{\varepsilon}{ 4\sqrt{n\ln n}}\binom{n/2-2\sqrt{n\ln n}}{j-22}\binom{n}{\lfloor n/2\rfloor}$ and the average degree is at least $\frac{\varepsilon}{ 8\sqrt{n\ln n}}\binom{n/2-2\sqrt{n\ln n}}{j-22}$ in $B$. By Observation \ref{core}, $B$ contains a subgraph $B'$ with minimum degree at least $\frac{\varepsilon}{ 16\sqrt{n\ln n}}\binom{n/2-2n\sqrt{n\ln n}}{j-22}$.

    Our strategy is as in Case I, we want to find a spider in $B'$ that corresponds to an $S^{k-2,\ell}$ in $\cF$ but now $\ell=\ell(n)=n^{1/3}$. This time, we require a little less from the color sets of the edges used: we want that for any edge $FF_2$ out of the $j$ colors, less than $j/k$ should appear in color sets of previously defined edges. The total number of colors used is at most $kn^{1/3}j$. So the number of sets having not enough new elements is at most
    \[
    \sum_{i=j/k}^{j}\binom{kn^{1/3}j}{i}\binom{n/2+2\sqrt{n\ln n}}{j-i}
    \]    
which is, by Lemma \ref{lemma} (2), negligible compared to the minimum degree of $B'$. So it is possible to pick the edges of the spider as required. As in the previous case, the spider is a weak copy of $S^{k-2,\ell}$ by definition of $B$. We claim that it is a strong copy of $S^{k-2,\ell}$. Suppose $F$ is in the $\cF$-part of $B$ and $F_2$ is in the $\cF_2$ part of $B$ and they both belong to the spider, but $FF_2$ is not an edge of the spider. If $F_2$ is closer to the center on the same leg as $F$ or if they are on distinct legs, then there is a step after $F_2$ is defined when $j-j/k$ new colors are added to the leg of $F$, and later at most $j/k$ old colors are removed in each step and there are at most $k/2-1$ steps when we remove colors, so $F$ will contain at least $j/2$ new colors compared to $F_2$. Similarly, if $F$ is closer to the center of the spider with $F_2$ on the same leg, then after adding $F$, there is a step when $j-j/k$ colors are removed, and then colors are added in at most $k/2-1$ steps, at most $j/k$ old colors in each such step, so $F$ again will have at least $j/2$ colors not in $F_2$.

    Let again $F^1,F^2, \dots, F^\ell$ be the sets corresponding to the leaves of the $S^{k-2,\ell}$. By the argument above, for any $1\le a,b\le \ell$, if $C^a$ denotes the color set of the edge incident to $F^a$, then $|C^a\setminus F^b|\ge j/2$. The number of subsets of $F^a$ of size $|F^a|-j$ that contain more than $j/k$ elements of $C^a$ is \[\sum_{i=j/k}^j\binom{|C^a|}{i}\binom{|F^a\setminus C^a|}{j-i}\le 2\binom{j}{j/k}\binom{n/2+2\sqrt{n\ln n}}{j-j/k},\]
    which by the calculation of Lemma \ref{lemma} (2), has a lower order of magnitude compared to the minimum degree in $B$. So every $F^a$ has at least $\frac{\varepsilon'}{\sqrt{n\ln n}}\binom{n/2-2\sqrt{n\ln n}}{j-22}$ incident edges in $B$ with at most $j/k$ colors used in the spider. As in Case I, let $G$ denote the intersection of all sets in the spider. As the total number of colors used is at most $kn^{1/3}j$,  we obtain $|G|\ge n/2-kn^{1/3}j$. So for every $a$ there exists a set $Y^a\subset G$ of $n^{2/3}$ elements that appear as colors on some edge incident to $F^a$. Indeed, if not then there exists $B^a\subset G$ with $|B^a|\le n^{2/3}$ such that all colors appearing on some edge incident to $F^a$ are in $F^a\cup B^a$, and therefore the number of such edges is at most $\binom{kn^{/3}j+2\sqrt{n\ln n}+n^{2/3}}{j}$ which, by Lemma \ref{lemma} (3), is negligible compared $\frac{\varepsilon'}{\sqrt{n\ln n}}\binom{n/2-2\sqrt{n\ln n}}{j-22}$. This contradiction proves the existence of $Y^a$. As $\ell=n^{1/3}$ there exist $a,b$ and $y$ with $y\in Y^a\cap Y^b$. Thus there exist $G^a,G^b$ with $y\notin G^a\subseteq F^a,y\notin G^b\subseteq F^b$, $|C^a\cap G^a|,|C^b\cap G^b|\le j/k$. Then $G^a,G^b$ and the legs of $F^a$ and $F^b$ form the needed copy of $P_{2k-1}$.
\end{proof}

\bigskip

We now turn our attention to Theorem \ref{specialtree}. As the proof very closely follows the proof in \cite{BJ}, let us briefly sketch their approach to embed a tree poset $T$ of height $k$ into any family $\cF\subseteq \tilde{B}_n$ of size at least $(k-1+\varepsilon)\binom{n}{\lfloor n/2\rfloor}$. First, following \cite{B}, they define a sequence $T\subseteq T_1\supset T_2\supset \dots \supset T_\ell$ of tree posets such that $T_\ell$ is a chain and $T_i\setminus T_{i+1}$ is a chain interval (see the definition below). Then a sequence of $\cG_1\supset \cG_2\supset \dots \supset \cG_\ell$ of families of \textit{good} $k$-chains of $\cF$ is defined, where good approximately means extendable for any small obstacles. Finally, the chain $T_\ell$ is embedded into $\cF$ using a $k$-chain from $\cG_\ell$ and any $T_i\setminus T_{i+1}$ is embedded using a $k$-chain from $\cG_i$ as $\cG_i$ is good thus extendable with respect to $\cG_{i+1}$. For weak copies of $T$, the good property is not necessary, but when one looks for strong copies of $T$ one has to make sure that sets $F_t$ from newly used chains are not in relation with sets $F_{t'}$ of already used chains if $t,t'$ are incomparable in $T$. As the number of already embedded elements (the potential obstacles) is at most $|T|-1$, the goodness condition ensures that one can extend the embedding for a strong copy of $T$. In the proof of Theorem \ref{specialtree}, we will apply the same procedure, but with one extra obstacle and this new obstacle might belong to $\tilde{B}_n^T\setminus \tilde{B}_n$. So the change is not very substantial for this part of the proof, but we need to be a bit more careful, when defining the sequence $T_1\supset \dots \supset T_\ell$. For this reason, we do not include all details. The proof of Theorem \ref{goodT} is basically identical to that of Theorem 5.1 of \cite{BJ} with constants adjusted. We do include its proof as an appendix to the version \cite{Pvers} that we upload to the author's webpage.

\medskip

We start describing the preliminaries by addressing $T$. A poset $P$ is \textit{$k$-saturated} if $h(P)=k$ and all maximal chains have length $k$.

\begin{lem}[Bukh \cite{B}]\label{saturated}
    Any poset $P$ of height $h$ is a strong subposet of an $h$-saturated poset $P'$. Moreover, any tree poset $T$ of height $h$ is a strong subposet of an $h$-saturated tree poset $T'$.
\end{lem}

\begin{proof}[Sketch of the proof.]
    Consider the canonical decomposition $\cup_{i=1}^hA_i=P$ that we obtain by setting $A_1$ to be the antichain of all minimal elements of $P$, and for any $1<j\le h$ we let $A_j$ be the set of minimal elements of $P\setminus \cup_{i=1}^{j-1}A_i$. 
    
    \begin{itemize}
        \item 
        If $q\in A_j$ is a maximal element of $P$, then add $h-j$ elements $q_1,q_2,\dots,q_{h-j}$ forming a chain $q\leqslant q_1\leqslant q_2\leqslant \dots \leqslant q_{h-j}$.
        \item 
        If $\overrightarrow{pq}$ is an arc of the oriented Hasse diagram of $P$ with $p\in A_i,q\in A_j$, then add $j-i-1$ elements $r_1,r_2,\dots,r_{j-i-1}$ forming a chain $p\leqslant r_1\leqslant\dots \leqslant r_{j-i-1}\leqslant q$.
    \end{itemize}
\end{proof}

If $p\leqslant_P q$, then the interval $[p,q]$ is $\{z:p\leqslant_P z\leqslant_P q\}$. A \textit{chain interval} is an interval that induces a chain in $P$. The \textit{poset distance} of $p,q$ in $P$ is the minimum $k$ such that there exist vertices $p=v_0,v_1,v_2,\dots, v_k=q$ such that $v_i$ and $v_{i+1}$ are comparable in $P$ for all $i=0,1,\dots,k-1$. The next lemma is from \cite{B}, the moreover part is not explicit there, but comes from its proof.

\begin{lem}[Bukh \cite{B}]\label{interval}
    If $T$ is a $k$-saturated tree poset that is not a chain, then there exist a leaf $v$ and a chain interval $I=[v,u]$ or $I=[u,v]$ such that $T\setminus I'$ is a $k$-saturated tree poset, where $I'=I\setminus \{u\}$. Moreover, $v$ is a leaf such that there exists $w\in T$ with $vw$ maximizing poset distance.
\end{lem}

We will need the following simple lemma on the possible size of unions and intersections of sets that are used when embedding $T$. These are the possible extra obstacles mentioned earlier.

\begin{lem}\label{tilde}
    Suppose $G_1,G_2,\dots,G_h \in \tilde{B}_n$ with $h \le |T|$ are such that their comparability graph is connected. Then $\cup_{i=1}^hG_i\in \tilde{B}_n^T$.
\end{lem}

\begin{proof}
    We can assume that the indexing of the $G_i$s is such that for any $1\le j \le h$ the comparability graph of $G_1,G_2, \dots,G_j$ is connected. Therefore there exists $i<j$ such that $G_i,G_j$ are comparable. If $G_j \subseteq G_i$, then $G_j$ does not add new element to the union of the $G_i$s, while if $G_i\subseteq G_j$, then $G_j$ adds at most $|G_j\setminus G_i|\le 4\sqrt{n\ln n}$ elements to the union of the $G_i$. Thus $|\cup_{i=1}^hG_i|\le n/2+2\sqrt{n \ln n}+4(h-1)\sqrt{n\ln n}\le n/2+4|T|\sqrt{n \ln n}$ as claimed.
\end{proof}

The following definitions are all from \cite{BJ}.
For a family $\cG \subseteq 2^{[n]}$ of sets, we define $\cD(\cG)=\cup_{G\in \cG}\cD(G)$ and $\cU(\cG)=\cup_{G\in \cG}\cU(G)$. For a set $G \in \tilde{B}_n$ and a family $\cS \subseteq \tilde{B}_n^T$ with $\cU(G)\cap \cS=\emptyset$ (in \cite{BJ} the family of obstacles satisfied $\cS\subseteq \tilde{B}_n$, but for the proof of Theorem \ref{specialtree}, we will have to allow $\cS\subseteq \tilde{B}_n^T$), we define the \textit{forbidden neighborhood below and above $G$} as
\[
\cD^*(G,\cS)=[(\cD(G)\setminus \{G\})\cap (\cU(\cS)\cup \cD(\cS))]\cap \tilde{B}_n
\]
and
\[
\cU^*(G,\cS)=[(\cU(G)\setminus \{G\})\cap (\cU(\cS)\cup \cD(\cS))]\cap \tilde{B}_n.
\]
Let $\mathbf{C}_n$ denote the set of all $n!$ maximal chains in $[n]$.  A pair $(\cC,\cQ)$ is a \textit{$k$-marked chain} with markers in $\cF$ if $\cC\in \bC_n$, $\cQ\subseteq \cC\cap \cF$ and $|\cQ|=k$. For a set $\cL$ of $k$-marked chains, a set $G$, and an integer $1 \le d \le k$, we write $\cL(G,d)=\{(\cC,\cQ)\in \cL: G ~\text{is the $d$th member of $\cQ$ from top}\}$.

\medskip

The next definitions are a little altered compared to their original version in \cite{BJ}. There, the witness family $\cS$ was from $\tilde{B}_n$, while here it is from $\tilde{B}_n^T$. This is because our additional obstacle is a union of some of the sets of a copy of $T$ in $\tilde{B}_n^T$, and then Lemma \ref{tilde} guarantees that this new obstacle lies within $\tilde{B}_n^T$.

\smallskip

A set $G\in \tilde{B}_n$ is \textit{$(d,T)$-lower bad relative} to $\cL$ if there exists a family $\cS\subseteq \tilde{B}_n^T$ with $\cU(G)\cap \cS=\emptyset$ and $|\cS|\le |T|$ such that 
\[
\cL(G,d)\neq \emptyset \hskip 1truecm \text{and}\hskip 0.5truecm \forall (\cC,\cQ)\in \cL(G,d), \hskip 0.3truecm \cQ\cap \cD^*(G,\cS)\neq \emptyset.
\]
Similarly, $G\in \tilde{B}_n$ is \textit{$(d,T)$-upper bad relative to $\cL$ }if there exists a family $\cS\subseteq \tilde{B}_n^T$ with $\cD(G)\cap \cS=\emptyset$ and $|\cS|\le |T|$ such that 
\[
\cL(G,d)\neq \emptyset \hskip 1truecm \text{and}\hskip 0.5truecm \forall (\cC,\cQ)\in \cL(G,d), \hskip 0.3truecm \cQ\cap \cU^*(G,\cS)\neq \emptyset.
\]
The family $\cS$ in the above definitions is called a \textit{$(d,T)$-lower (upper) witness} of $G$.

We say
that $G$ is \textit{$(d,T)$-lower-bad relative to $\cC$ and $\cL$} if $G$ is $(d,T)$-lower-bad relative to $\cL$ and there exists at
least one $\cQ$ such that $(\cC,\cQ) \in \cL(G, d)$. We say that $G$ is \textit{$(d,T)$-upper-bad relative to $\cC$ and $\cL$} if $G$ is
\textit{$(d,T)$-upper-bad relative to $\cL$} and there exists at least one $\cQ$ such that $(\cC,\cQ) \in \cL(G, d)$. A $k$-marked
chain $(\cC,\cQ)$ is\textit{ good relative to $\cL$} if $\cQ$ does not contain a vertex $G$ that is either $(d,T)$-lower-bad or
$(d,T)$-upper-bad relative to $\cC$ and $\cL$ for any $1\le d\le k$. 

\begin{prop}[Equivalent of Proposition 4.1 in \cite{BJ}]\label{güdT}
    Let $(\cC,\cQ)$ be a member of $\cL$ that is good relative to $\cL$, and let $G\in \cQ$. Suppose $G$ is the $d$th set of $\cQ$. Then for any family $\cS$ of at most $|T|$ sets of $\tilde{B}^T_n$, where $\cS \cap \cU(G) = \emptyset$, there exists a member $(\cC, \cQ) \in \cL(G, d)$ such that $\cC$ is disjoint from $\cD^*(G, \cS)$. For any family $\cS$ of at most $|T|$ sets of $\tilde{B}^T_n$, where $\cS \cap \cD(G) = \emptyset$, there exists a member $(\cC, \cQ) \in \cL(G, d)$ such that $\cC$ is disjoint from $\cU^*(G, \cS)$.
\end{prop}

\begin{proof}
    By our assumptions $(\cC,\cQ) \in \cL(G, d)$ and $G$ is not $(d,T)$-lower-bad or $(d,T)$-upper-bad
relative to $\cL$; otherwise $G$ would be either $(d,T)$-lower-bad or $(d,T)$-upper-bad relative to $\cC$ and $\cL$, contradicting $(\cC,Q)$ being good relative to $\cL$. So, there is no $(d,T)$-lower witness of $G$ or $(d,T)$-upper witness of $G$ of size at most $|T|$ and the claim follows.
\end{proof}

As mentioned before, we do not present the proof of the next theorem here only in the version \cite{Pvers} uploaded to the author's website.

\begin{thm}[Equivalent of Theorem 5.1. in \cite{BJ}]\label{goodT}\
Let $k,t\ge 2$ be integers and let $\varepsilon$ be a fixed positive real. 
Suppose that $\cF \subseteq \tilde{B}_n$ is a family with $|\cF| \ge (k - 1 +\varepsilon )\binom{n}{\lfloor n/2\rfloor}$. For each $\cC \in \bC_n$, let $Y (\cC)$ denote the set of members of $\cF$ contained in $\cC$. If $n$ is large enough with respect to $k,t$, and $\varepsilon$, then there exist functions $X_1,\dots,X_t$ from $\cC_n$ to
$2^{\cF}$ satisfying the following.
    \begin{enumerate}
        \item 
        For all $\cC \in \bC_n$, $X_1(\cC) = Y (\cC)$.
        \item
For all $1 \le i \le t-1$ and $\cC \in \bC_n$, $X_{i+1}(\cC) \subseteq X_i(\cC)$, and if $X_{i+1}(\cC) \neq \emptyset$ then
\[
\frac{|X_{i+1}(\cC)|}{|X_i(\cC)|}\ge  1 -\frac{1}{4kt}.
\]
    \item 
    For all $1\le i \le t$, the family of $k$-marked chains $\cL_i$ with markers in $\cF$, defined by $\cL_i =\{(\cC,Q) : \cC \in \bC_n, Q \in \binom{X_i(\cC)}{k}\}$, satisfies
\[
|\cL_i|\ge \varepsilon  (n!/k)(1 - \frac{i}{2t}).
\]
\item 
For all $1\le i \le  t -1$, every member of $\cL_{i+1}$ is good relative to $\cL_i$. 
    \end{enumerate}
\end{thm}

With all the above preparation, we are ready to prove Theorem \ref{specialtree}.

\begin{proof}[Proof of Theorem \ref{specialtree}] Let $T$ be a tree poset of height $k+1$ such that $m$ is a maximal element of $T$ with all chains of length $k+1$ in $T$ containing $m$. Let $m_1,m_2,\dots,m_a$ denote the neighbors of $m$ in $H(T)$. Let $T_0$ be the tree obtained from $T$ by reversing all arcs $\overrightarrow{m_im}$ in $\overrightarrow{H}(T)$ and adding vertices $v_1,v_2,\dots,v_{k-2}$ forming a chain $v_1\leqslant v_2\leqslant \dots \leqslant v_{k-2}\leqslant m$. Clearly, $h(T_0)=k$ and $T\setminus \{m\}=T_0\setminus \{m,v_1,v_2,\dots,v_{k-2}\}$. Let $T_0'$ be the $k$-saturated tree poset obtained from $T_0$ by Lemma \ref{saturated}. Observe that for any pair $p,q\in \cD:=\cup_{i=1}^a\cD_{T_0'}(m_i)$ the poset distance in $T_0'$ is at most 4, and if $a\ge 2$, then for any $v\in T_0'\setminus \cD$ there exists $p\in \cD$ such that their poset distance is larger than 4. If $a=1$, then the poset distance of pairs in $\cD$ is at most 2. In both cases, one can apply Lemma \ref{interval} to obtain a sequence $T_0',T_1',\dots,T_\ell'$ such that
\begin{itemize}
    \item 
    all $T_i'$ are $k$-saturated,
    \item 
    $T'_{i-1}\setminus T'_i$ is $[u_i,v_i]\setminus \{u_i\}$ or $[u_i,v_i]\setminus \{v_i\}$ where $[u_i,v_i]$ is a chain interval, 
    \item 
    $T'_\ell$ is a chain, 
    \item 
    and $T'_{i_0}=\cD$ for some $i_0$.
\end{itemize} 
 Let $\cF\subseteq \tilde{B}_n$ be a family of size at least $(k-1+\varepsilon)\binom{n}{\lfloor n/2\rfloor}$. Applying Theorem \ref{goodT}, we obtain a sequence of nested families of $k$-marked chains $\cL_0,\cL_1,\dots,\cL_\ell$. By property (3) of Theorem \ref{goodT}, $\cL_\ell$ is non-empty, so taking the $\cQ_\ell$-part of any $k$-marked chain $(\cC_\ell,\cQ_\ell)\in \cL_\ell$ embeds $T'_\ell$ into $\tilde{B}_n$ and for any $Q\in \cQ$ at the $d$-level of $T'_\ell$ we have $\cL_\ell(Q,d)\neq \emptyset$. As in \cite{BJ}, we proceed by backward induction to embed $T'_i$ for all $i<\ell$ with the extra conditions that if $Q_i$ plays the role of an element of $T'_i$ at level $d$ from top, then $\cL_i(Q_i,d)\neq \emptyset$ and also that if $p\in T'_0\setminus \cD$, then $G_p$ is not comparable to $G=\cup_{j=1}^aG_{m_j}$. Note that if $p\in T'_0\setminus \cD$, then $p\in T'_i\setminus T'_{i+1}$ for some $i<i_0$ and so $m_1,m_2,\dots,m_a$ are already embedded. Suppose we have managed to do this for $T'_{i+1},\dots,T'_\ell$ and let $\cG_j\subseteq \cF$ be the copy of $T'_j$ obtained in the process. We know $T'_{i+1}=T'_i\setminus (I_{i+1}\setminus u_{i+1})$ where $I_{i+1}=[u_{i+1},v_{i+1}]$ or $I_{i+1}=[v_{i+1},u_{i+1}]$ $v_{i+1}$ being at level $1$ or $k$ from top. Without loss of generality, we may assume the former case and assume $u_{+1}i$ is at level $d$ from top. Then by the inductive hypothesis $\cL_{i+1}(Q_{i+1},d)\neq \emptyset$ for the set $Q_{i+1}\in \cF$ playing the role of $u_{i+1}$. If $i\ge i_0$, then let $\cS_i=\cG_{i+1}\setminus \cD_{\cG_{i+1}}(Q_{i+1})$, while if $i<i_0$, then set $\cS_i=\cG_{i+1}\setminus \cD_{\cG_{i+1}}(Q_{i+1})\cup \{G\}$. As $T'_0\setminus T'_1\neq \emptyset$, we have $|\cS_i|\le |T'_0|=t$ in both cases. By Proposition \ref{güdT}, there exists a member $(\cC_i,\cQ_{i+1})\in \cL_{i+1}(Q_{i+1},d)$ with $\cC_i$ being disjoint with $\cU^*(Q_{i+1},\cS_i)$. We can embed $I_{i+1}\setminus \{u_{i+1}\}$ with the part of $\cQ_i$ below $Q$. As these newly added sets are not in the forbidden neighborhood $\cU^*(Q_{i+1},\cS_i)$, together with $\cG_{i+1}$, they form a copy $\cG_i$ of $T'_i$ such that the extra condition of not being comparable to $G$ hold. Also, the $\cL_i$s are nested, $(\cC_i,\cQ_{i+1})$ shows that the other extra condition holds.
\end{proof}

\subsection{Posets not containing  weak copies of $2C_2$}

\begin{proof}[Proof of Proposition \ref{noC2}]
For positive integers $a<b$, we write $[a,b]=\{a,a+1,\dots,b\}$ and we use $[b]=[1,b]$. For the lower bound of (1), consider the following construction: for $j$ with $2 \le j \le s-1$ take the chain $\cC_j=\{\{j\},[j,j+1],[j,j+2],\dots,[j,j-3],[j,j-2]\}$, where $[a,b]$ with $b<a$ denotes the set $[a,n]\cup [b]$. Also, let $\cC_1$ be the chain $\{[1],[2],\dots,[s-2],[s-2]\cup \{s\},[s-2]\cup[s,s+1],\dots,[s-2]\cup [s,n]\}$. Color all sets in $\{\emptyset,[n]\}\cup\bigcup_{i=1}^{s-1}\cC_i$ with distinct colors and color all remaining sets with the color of $[n]$. This coloring uses $(s-1)(n-1)+2$ colors and does not admit a strong rainbow copy of $\wedge_s$ as $\emptyset$ is not a member of any strong copy. Also, by construction, any set $C\in \cC_i$ does not contain any set $C'\in \cC_{i-1}$ and any set $C\in \cC_1$ does not contain any set $C'\in \cC_{s-1}$ and so the maximum rainbow antichain that $C$ contains has size $s-1$: $s-2$ sets from $\cC_j$s plus one colored as $[n]$. Finally, any set colored as $[n]$ contains rainbow antichains only from $\cup_{i=1}^{s-1}\cC_i$, so at most $s-1$ sets. This shows $\ar(n,\wedge_s)\ge (s-1)(n-1)+2$. For the upper bound, observe that if a coloring of $2^{[n]}$ uses more than $(s-1)(n-1)+2$ colors, then, by Dilworth's theorem, there is a rainbow antichain of size $s$ that uses colors differing from that of $\emptyset$ and that of $[n]$.

For the lower bound of (2) consider $k-2$ pairwise internally disjoint maximal chains and color all sets in their union by distinct colors and use PURPLE for all other sets. As any antichain can contain at most one set from each chain, the maximum size of a rainbow antichain is $k-1$. This proves $\ar(n,A_k)\ge 3+(k-2)(n-1)$. For the upper bound, observe that any level $\binom{[n]}{j}$ can use at most $k-1$ colors apart from that of $[n]$ (and that of $\emptyset$) as otherwise would have a rainbow $A_k$. Also, any level $\binom{[n]}{j}$ can use at most $k-2$ colors not used in $\binom{[n]}{k}$. Indeed, if $k<j$ and $F_1,F_2,\dots,F_{k-1}$ are sets of size $j$ using colors unused by any set of size $k$, then pick $x_i\notin F_i$ for all $1\le i \le k-1$ and a $k$-set containing all $x_i$s (they may coincide) forms a rainbow $A_k$ with the $F_i$s. If $j<k$, then for such $F_i$s consider $y_i \in F_i$ and a $k$-subset $F$ of $[n]\setminus \{y_1,y_2,\dots,y_{k-1}\}$ (here we use $k\le n-k$). So the total number of colors used is $2+(k-1)+(n-2)(k-2)$ (two for the colors of $[n],\emptyset$, then $k-1$ for the colors used in $\binom{[n]}{k}$, and $k-2$ additional colors for all other $n-2$ levels). This completes the proof.

To see the upper bound of (3), observe that for any $j\le n-2$, by counting pairs $(F,x)$ $x\notin F$, if $\cF\subseteq \binom{[n]}{j}$ contains more than $\frac{sn}{2}$ sets, then there exists a subfamily $\cF'\subseteq \cF$ such that $|\cF'|\ge s+1$ and $|\cup_{F\in \cF'}F|<n$. Let $m$ be the size of the smallest possible $|\cup_{F\in \cF'}F|$ with this property. Using the minimality of $m$, the same reasoning shows that $|\cF'|\le \frac{sm}{2}$. So if $|\cF|>\frac{sn}{2}+k+1$, then there exist $\cF'$ and $\cF''$ with $|\cF'|=s+1,|\cF''|=k+1$ such that $F''\not\subseteq \cup_{F'\in \cF'}F'$ for all $F''\in \cF''$. Let us apply this statement to the family $\cF$ that we obtain by picking a set from each color meeting $\binom{[n]}{j}$. Now, adding $M=\cup_{F'\in \cF'}F'$ and possibly removing the $j$-set in $\cF'\cup \cF''$ with the color of $M$ yields a rainbow copy of $\wedge_s+A_k$. So, if a coloring avoids such a copy, then every level uses at most $\frac{sn}{2}+k+1$ colors. As there are $n-1$ levels and $\emptyset,[n]$, the result follows.
\end{proof}

\noindent \textbf{Acknowledgement}. We would like to thank the anonymous referees for their thorough reading and their remarks that helped the
presentation of the results. Also, we thank Maria Axenovich for letting us know about Jacob Manske’s PhD thesis \cite{M}.

\end{document}